\documentclass[journal,twoside,web]{ieeecolor}
\usepackage{lcsys}
\let\labelindent\relax

\usepackage{graphicx}
\usepackage{epsfig}
\usepackage{mathptmx}
\usepackage{times}
\usepackage{amsmath} 
\usepackage{amsthm}
\usepackage{amssymb}
\usepackage{tcolorbox}
\usepackage{cite}
\usepackage[ruled,vlined]{algorithm2e}
\usepackage[colorlinks = true, linkcolor = blue, citecolor =
blue]{hyperref}
\usepackage{enumitem}
\usepackage{subcaption}
\usepackage[mathscr]{euscript}
\usepackage{xcolor}
\usepackage{tabularx}

\newcommand\new[1]{{\color{black}#1}}

\newtheorem{thm}{Theorem}
\newtheorem{prop}[thm]{Proposition}
\newtheorem{lem}[thm]{Lemma}

\newtheorem{rem}[thm]{Remark}

\newcommand{\ldef}{:=}

\newcommand{\Mc}[1]{\mathcal{#1}}
\newcommand{\real}{\ensuremath{\mathbb{R}}}

\newcommand{\realp}{\ensuremath{\mathbb{R}_{>0}}}
\newcommand{\realz}{\ensuremath{\mathbb{R}_{\ge 0}}}

\newcommand{\nat}{{\mathbb{N}}}
\newcommand{\natz}{{\mathbb{N}}_0}

\newcommand{\dd}{\mathrm{d}}

\newcommand{\norm}[1]{\left\lVert #1 \right\rVert}

\DeclareMathOperator*{\argmin}{arg\,min}

\newcommand{\tth}{^\text{th}}
\newcommand{\param}{\textbf{a}}
\newcommand{\sys}{\mathscr{S}}

\newcommand{\thmtitle}[1]{\mbox{}\textit{(#1).}}

\renewcommand{\theenumi}{(\arabic{enumi})}

\newcommand{\remend}{\relax\ifmmode\else\unskip\hfill\fi\hbox{$\bullet$}}

\begin{document}
	
	\title{\LARGE \bf
		Event-Triggered Parameterized Control of Nonlinear Systems
	}

	\author{Anusree Rajan and Pavankumar Tallapragada
		\thanks{The authors are with the Department of Electrical 
			Engineering, Indian Institute of Science.
			{\tt\small \{anusreerajan, pavant\}@iisc.ac.in}}
	}
	
	\maketitle
	\thispagestyle{empty}
	\pagestyle{empty}

	%%%%%%%%%%%%%%%%%%%%%%%%%%%%%%%%%%%%%%%%%%%%%%%%%%%%%%%%%%%%%%%%%%%%%%%%%%%%%%%%
\begin{abstract}
This paper deals with event-triggered parameterized control (ETPC) of nonlinear systems with external disturbances. In 
this control method, between two successive events, each control input 
to the plant is a linear combination of a set of linearly independent 
scalar functions. At each event, the controller updates the coefficients of the parameterized 
control input so as to 
minimize the error in approximating a continuous time control signal and communicates the same to the actuator. We design an 
event-triggering rule \new{(ETR)} that guarantees global uniform ultimate 
boundedness of trajectories of the closed loop system. We also ensure 
the absence of Zeno behavior by showing the existence of a uniform 
positive lower bound on the inter-event times \new{(IETs)}. We illustrate our 
results through numerical examples.
\end{abstract}

	
	%%%%%%%%%%%%%%%%%%%%%%%%%%%%%%%%%%%%%%%%%%%%%%%%%%%%%%%%%%%%%%%%%%%%%%%%%%%%%%%%
\section{INTRODUCTION}
\IEEEPARstart{E}{vent-triggered} control (ETC) is a commonly used control method in 
applications with resource constraints. Most of the ETC literature designs zero-order-hold (ZOH) sampled-data 
controllers. However, in many common communication protocols, including 
TCP and UDP~\cite{DH:2020}, there is a minimum packet size. Thus, ZOH 
control may lead to an increase in the total number of communication 
instances due to under utilization of each packet. With this 
motivation, in this paper, we propose a non-ZOH control method and 
design it for control of nonlinear systems with external disturbances. 

\subsection{Literature Review}
An introduction to ETC and an overview of the literature on it can be 
found in~\cite{ PT:2007, WH:2012, ML:2010, 
DT-SH:2017-book}. Typically, in ETC 
and in the closely related self-triggered control~\cite{AA:2010} and 
periodic \new{ETC}~\cite{WH:2013}, control input is 
applied in ZOH fashion, i.e., the control input to the plant is held 
constant between any two successive events. There are some exceptions 
though. For example, in model-based ETC 
\new{(MB-ETC)}~\cite{EG-PA:2013, 
MH-FD:2013, HZ-etal:2016, ZC-etal:2021, 
LZ-etal:2021}, both the controller and the actuator use identical 
copies of a model of the system, whose states are updated synchronously 
in an event-triggered manner. The model generates a time-varying 
control input even between two successive events. In 
event/self-triggered model predictive control~\cite{HL-YS:2014, 
FD-MH-FA:2017, HL-etal:2018}, the actuator applies a part of an optimal 
control trajectory, which is generated by solving a finite horizon 
optimal control problem at each triggering instant. Recent studies in~\cite{AL-JS:2023,KH-etal:2017} show that  communication 
resources can be utilized more efficiently by transmitting only some of 
the samples of the generated control trajectory to the actuator, based 
on which a sampled data first-order-hold (FOH) control input is 
applied. 
In event-triggered dead-beat control 
\new{(ET-DBC)}~\cite{BD-etal:2017}, a control input sequence is 
transmitted to the actuator in an event-triggered 
manner. The actuator stores this control sequence in a buffer and 
applies it till the next packet is received. In team-triggered 
control~\cite{CN-JC:2016,JL-etal:2021} each agent makes promises to 
its neighbors about their future states or controls and informs them 
if these promises are violated later. 

References~\cite{AC-CL:1978,PK:1987,JL-etal:2010} use generalized 
sampled-data hold functions (GSHF) in the control of linear 
time-invariant systems. The idea of GSHF is to periodically sample the 
output of the system and generate the control by means of a hold 
function applied to the resulting sequence. To the best of our 
knowledge, this idea was first explored in the context of 
ETC only in our recent work~\cite{AR-PT:2023}, in 
which we propose an event-triggered parameterized control (ETPC) method 
for stabilization of linear systems. In~\cite{AR-etal:2023}, we use a 
similar idea to design an event-triggered polynomial controller for 
trajectory tracking by unicycle robots.

\subsection{Contributions} \label{sec:contribs}

The contributions of this paper are given below:

\begin{itemize}
	\item We \new{propose an ETPC method}, for nonlinear systems with 
	external disturbances, that guarantees global uniform ultimate 
	boundedness of trajectories of the closed loop system and non-Zeno 
	\new{ inter-event times (IETs)}.
	\item Our approach requires fewer 
	communication packets compared to ZOH or FOH control, as our method can be fine tuned to utilize the full payload of each packet. 
	\item Compared to \new{MB-ETC}, our method 
	requires less computational resources at the actuator and also 
	provides greater privacy and security. 
	\item Compared to the \new{ET-}MPC or \new{ET-DBC} method, at each event, our proposed method requires only a limited 
	number of parameters to be sent irrespective of the time duration of 
	the signal. 
	\item In this paper, we generalize the control method proposed in our 
	previous work~\cite{AR-PT:2023} to nonlinear control settings with 
	external disturbances. In~\cite{AR-PT:2023}, the analysis 
	heavily relied on linear systems theory and closed form expressions 
	for the solutions. This approach is not applicable to nonlinear 
	systems, wherein there are also many non-trivial technicalities that 
	have to be taken care of. \new{We also allow for a wider choice for 
	the set of basis functions of the parameterized control input.}
	\item Our recent work~\cite{AR-etal:2023} 
	considers a similar control method for the trajectory tracking by 
	unicycle robots. In the current paper, we consider a more generalized 
	problem setup and also incorporate the effect of external 
	disturbances.
\end{itemize}

\subsection{Notation}

Let $\real$, $\realz$ and $\realp$ denote the set of all real numbers, 
the set of non-negative real numbers and the set of positive real 
numbers, respectively.
Let $\nat$ and $\natz$ denote the set of all positive and non-negative 
integers, respectively. For any $x \in \real^n$, $\norm{x}$ denotes the 
Euclidean norm. A continuous function $\alpha: [0,\infty) 
\to  [0,\infty)$ is said to be of class $\Mc{K}_{\infty}$ if it is 
strictly increasing, $\alpha(0)=0$ and $\alpha(r) \to \infty
$ as $r \to \infty$. For any right continuous function $f: \real_{\ge 
0} \to \real^n$ and $t \ge 0$, $f(t^{+}) \ldef \lim\limits_{s \to 
t^{+}} f(s)$. 
For any continuous function $f:\real \to \real$, 
  \begin{equation*}
    D^+f(t) \ldef \limsup_{h \rightarrow 0^+} \frac{f(t+h)-f(t)}{h}
  \end{equation*}
  denotes the upper right hand derivative of $f$. For any two functions 
  $v, w : [0, \infty] 
\rightarrow \real_{\ge 0}$, let
\begin{equation*}
\langle v, w \rangle_{T} := \int_{0}^{T} v( \tau ) w( \tau ) 
\dd \tau.
\end{equation*}
Note that $\langle v, w \rangle_T$ is the inner product of the 
functions $v(\tau)$ and $w(\tau)$ restricted to 
the domain of $\tau$ to $[0, T]$. 

\vspace{0.5em}

\section{PROBLEM SETUP}	\label{sec:problem_setup}

In this section, we present the system dynamics, the parameterized control law and the objective of this paper.

\subsection*{System Dynamics and Control Law}

Consider a nonlinear system with external disturbance,
\begin{equation}\label{eq:sys}
\dot{x} = f(x,u,d), \quad \forall t \ge t_0 = 0,
\end{equation}
where $x\in \real^n$, $u \in \real^m$, $d \in \real^q$  and $t \in 
\realz$, respectively, denote the system state, the control input, the 
external disturbance and the time. In this paper, we consider 
event-triggered sampled data non-ZOH control. We call our proposed method 
\emph{event-triggered parametrized control (ETPC)}.

Specifically, consider a set of functions
\begin{equation*}
  \Phi := \left\{ \phi_j: [0, \infty] \to \real 
  \right\}_{j=0}^p .
\end{equation*}
We let the $i\tth$ control input, for $i \in \{1,2,\ldots,m\}$, between 
two successive events be
\begin{equation*}
  u_i(t_k + \tau) = g(\param_i(k),\tau) := \sum_{j=0}^{p} 
  a_{ji}(k) \phi_j(\tau), \ \forall \tau \in [0, t_{k+1} - t_k) .
\end{equation*}
Here $(t_k)_{k \in \natz}$ is the sequence of communication time 
instants, which are determined in an event-triggered manner. At 
$t_k$, the controller updates the coefficients of the parameterized 
 control input, $\param(k) \ldef [a_{ji}(k)]\in \real^{(p+1) \times 
 m}$, and communicates them to the actuator. Letting 
 $\param_i(k)$ denote the $i\tth$ column of a(k), and $\phi(\tau) \ldef 
 \begin{bmatrix}
   \phi_0(\tau) & \phi_1(\tau)& \ldots & \phi_p(\tau)
 \end{bmatrix}^\top$, we can write the control law as,
\begin{equation}\label{eq:control_law}
u(t_k + \tau) = \param^\top(k) \phi(\tau), \ \forall \tau \in [0, 
t_{k+1} - t_k) .
\end{equation}

The general configuration of the ETPC system is depicted in 
Figure~\ref{fig:ETPC_system}.
\begin{figure}[h]
	\centering
		\includegraphics[width=7cm]{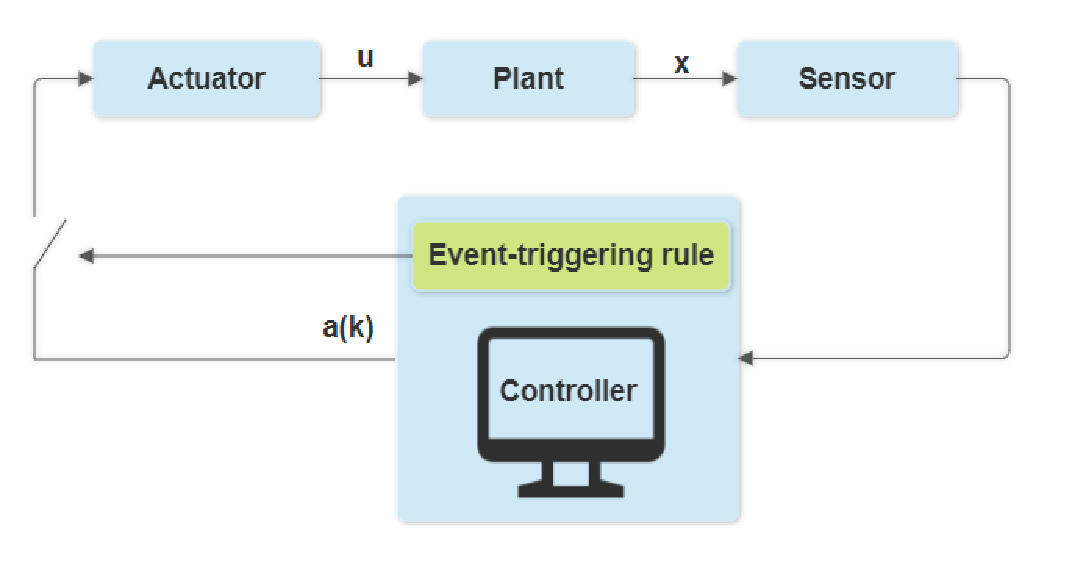}
		\caption{Event-triggered parameterized control configuration}
		\label{fig:ETPC_system}
\end{figure} 
Here, the system state is continuously available to the 
controller which has enough computational resources to evaluate the 
event-triggering condition and to update the coefficients of the 
control input at an event-triggering instant.

\subsection*{Assumptions}

We make the following assumptions throughout this 
paper.

\begin{enumerate}[resume, label=\textbf{(A\arabic*)},align=left]
   \item There exist $\gamma : \real^n \to \real^m$ and a continuously 
  differentiable Lyapunov-like function $V:\real^n \to 
  \real$ such that 
  \begin{equation*}
    \alpha_1(\norm{x}) \le V(x) \le \alpha_2(\norm{x}),
  \end{equation*}
  \begin{equation*}
    \frac{\partial V}{\partial x}f(x,\gamma(x)+e,d) \le 
    -\alpha_3(\norm{x})+\rho_1(\norm{e})+\rho_2(\norm{d})
  \end{equation*}
  where $\alpha_1(.)$, $\alpha_2(.)$, $\alpha_3(.)$, $\rho_1(.)$, 
  $\rho_2(.)$ are class $\Mc{K}_\infty$ functions and $e \in \real^m$, 
  $e \ldef u-\gamma(x)$ is the ``actuation error'' between 
  the control $u$ and $\gamma(x)$.
  \label{A:V}
\item $f(.)$ and $\gamma(.)$ are Lipschitz on compact sets, with $f(0)=0$ and $\gamma(0)=0$.
\label{A:sys}
  \item There exists $D\ge0$ such that $\norm{d(t)}\le D, \ \forall t 
  \ge t_0$.
  \label{A:d}
\end{enumerate}
Note that, Assumption~\ref{A:V} indicates that there exists a 
continuous-time feedback controller that makes the 
system~\eqref{eq:sys} input-to-state-stable (ISS) with respect to the 
``actuation error'' $e$ and the external disturbance $d$. 
Assumption~\ref{A:sys} is a common technical assumption in the 
literature on nonlinear systems. Finally, 
Assumption~\ref{A:d} means that the disturbance signal is uniformly 
upper bounded, which is again common in the literature. Throughout this 
paper, we make the following standing assumption regarding $\Phi$.

\begin{enumerate}[resume, label=\textbf{(A\arabic*)},align=left]
  \item Each function $\phi_j \in \Phi$ is continuously 
  differentiable and $\phi_0$ is a non-zero constant function. Let $T$ be a 
  fixed parameter and suppose $\Phi$ is a set 
  of linearly independent functions when restricted to $[0, T]$, i.e.,
%  \begin{equation*}
   $ \sum_{j=0}^{p}c_j\phi_j(t)=0,$ $\forall t \in [0, T]$ iff $c_j = 0,$ $\forall j \in \{0,1,\ldots,p\}$.
%  \end{equation*}
  \label{A:phis}
\end{enumerate}

\subsection*{Objective}

Our aim is to design a parameterized control 
law~\eqref{eq:control_law} and an event-triggering rule \new{(ETR)} for implicitly 
determining the communication instants $(t_k)_{k \in \natz}$ so that 
the trajectories of the closed loop system are globally uniformly 
ultimately bounded while ensuring a uniform positive lower bound on the 
\new{IET}s.

\section{DESIGN  AND ANALYSIS OF EVENT-TRIGGERED CONTROLLER}\label{sec:design}
In this section, we first design a parameterized control law and an 
\new{ETR} to achieve our objective. Then, we analyze the 
designed control system.

\subsection{Control Law}

The proposed control method is based on the idea of emulating a 
continuous time model based control signal using a 
parametrized time-varying signal as in~\eqref{eq:control_law}. In 
particular, consider the following model for some time horizon $T$,
\begin{equation}\label{eq:simulated_sys}
  \dot{\hat{x}}=f(\hat{x},\gamma(\hat{x}),0), \ \forall t \in [t_k, 
  t_k+T], \quad \hat{x}(t_k)=x(t_k), \ k \in \natz .
\end{equation}
Here $\hat{x}$ is the state of the model, which is the same 
as~\eqref{eq:sys} but with $u=\gamma(\hat{x})$ and $d=0$. The 
model state is 
reinitialized with $\hat{x}(t_k)=x(t_k)$ at each event time $t_k$ for 
$k \in \natz$. Now consider the open-loop control signals
\begin{equation*}
  \hat{u}_i( \tau ) \ldef \gamma_i( \hat{x}(t_k+\tau) ), \quad \forall 
  i \in \{1,2,\ldots,m\} ,
\end{equation*}
where $\gamma_i( \hat{x}(t_k+\tau) )$ is the $i^{th}$ component of 
$\gamma(\hat{x}(t_k + \tau))$. One way to potentially reduce the number 
of communication instances is to transmit the whole control signal 
$\hat{u}(\tau)$ for $\tau \in [0, T]$. For 
example, this is what is done in \new{ET-DBC}~\cite{BD-etal:2017} and \new{ET-}MPC~\cite{HL-YS:2014, FD-MH-FA:2017}. 
However, transmitting the whole control signal 
$\gamma(\hat{x}(t_k+\tau))$ for $\tau \in [0, T)$ in a communication 
packet at $t_k$ may be too costly.

So, in our proposed idea, we approximate $\hat{u}_i( \tau )$ 
for each $i$ in the linear span of $\Phi$. Specifically, we solve the 
following finite horizon optimization problem to determine the 
coefficients of the parameterized control signal~\eqref{eq:control_law} 
that is to be applied starting at $t_k$. For 
$i \in \{1,2,\ldots,m\}$,
\begin{equation}\label{eq:a_ik}
\begin{aligned}
\param_i(k) \in \argmin_{a \in \real^{p+1}} & 
\int_{0}^{T} |g(a,\tau)-\hat{u}_i(\tau)|^2d\tau,\\
\textrm{s.t.} \quad &  |g(a,0)-\hat{u}_i(0)| \le \eta(\norm{\hat{x}(t_k)})\\
\end{aligned}
\end{equation} 
for a function $\eta: \real_{\ge 0} \to \real_{\ge 0}$ with $\eta(0)=0$ 
and for a finite time horizon $T>0$. The function $\eta$ and the time 
horizon $T$ are to be designed. 
Note that, we require the signal $\hat{u}(\tau)$ for $\tau \in [0, T]$ to solve the optimization problem~\eqref{eq:a_ik}. 
This signal can be obtained by numerically simulating the $\hat{x}$ 
dynamics~\eqref{eq:simulated_sys}. 

\begin{rem}\thmtitle{Control input for $\tau > T$}
  With the parameters $\param(k)$ obtained by solving~\eqref{eq:a_ik}, 
  the control input applied by the actuator is as given 
  in~\eqref{eq:control_law}. Since $t_k$'s are 
  implicitly determined by an \new{ETR} online, it may 
  happen that $t_{k+1} - t_k > T$. However, even though we 
  find $\param(k)$ by using $\hat{u}(\tau)$ for $\tau \in [0, T]$, $g( 
  \param_i, \tau )$, for each $i$, is well defined $\forall\tau \in [0, 
  \infty]$.  Hence, the control input $u(t_k + \tau)$ for 
  $\tau$ is well defined for the entire interval $[t_k, t_{k+1})$ even 
  if $t_{k+1} - t_k > T$.
  \remend
\end{rem}

Note that, we can rewrite the 
optimization problem~\eqref{eq:a_ik} as,

\begin{equation}\label{eq:a_k}
\min_{a \in \real^{p+1}} J_i(a) \ := \frac{1}{2} a^\top 
\textbf{H} 
a - b_i^\top(k) a + c_i(k) , \ \textrm{s.t. } 
\ y^\top a-z_i(k) \le 0 ,
\end{equation}

where,
\begin{equation}\label{eq:hessian}
  \textbf{H} = 2\begin{bmatrix}
    \langle\phi_0,\phi_0\rangle_T & 	
    \langle\phi_0,\phi_1\rangle_T & \ldots 
    & \langle\phi_0,\phi_p\rangle_T\\[0.5em]
    \langle\phi_1,\phi_0\rangle_T & 	
    \langle\phi_1,\phi_1\rangle_T & \ldots 
    & \langle\phi_1,\phi_p\rangle_T\\
    \ldots & 	\ldots & \ldots & \ldots\\
    \langle\phi_p,\phi_0\rangle_T & 	
    \langle\phi_p,\phi_1\rangle_T & \ldots 
    & \langle\phi_p,\phi_p\rangle_T
  \end{bmatrix},
\end{equation}

\begin{align*}
  b_i^\top(k) &\ldef 2 \begin{bmatrix}
      \langle\hat{u}_i,\phi_0\rangle_T &      
      \cdots &
      \langle\hat{u}_i,\phi_p\rangle_T
    \end{bmatrix} , \ 
  c_i(k) \ldef \langle\hat{u}_i,\hat{u}_i\rangle_T
  \\
  y^\top &\ldef \begin{bmatrix}
    \phi^\top(0) \\ -\phi^\top(0)
  \end{bmatrix}, \quad
  z_i(k) = \begin{bmatrix}
  \hat{u}_i(0)+\eta(\norm{\hat{x}(t_k)}) \\
  -\hat{u}_i(0)+\eta(\norm{\hat{x}(t_k)})
\end{bmatrix} ,
\end{align*}
where recall that $\phi(\tau) = \begin{bmatrix}
  \phi_0(\tau) & \phi_1(\tau)& \ldots & \phi_p(\tau)
\end{bmatrix}^\top$.

\begin{prop}\label{prop:convexity}
	Problem~\eqref{eq:a_k} is a strictly convex optimization problem and 
	it is always feasible. Problem~\eqref{eq:a_k} always has exactly 
	one optimal solution.
\end{prop}
\begin{proof}
	The Hessian of $J_i(.)$ for all $i \in \{1,2,\ldots,m\}$, is
	$\textbf{H}$. Note that $\textbf{H}$ is twice the Gram matrix for 
	the functions in $\Phi$. Also, $\Phi$ is a set of 
	linearly independent functions when restricted to $[0,T]$. 
	Thus, we can say that $\textbf{H}$ is a positive definite 
	matrix. Hence, the cost function in~\eqref{eq:a_k} is strictly 
	convex. Note that the only constraints in the optimization 
	problem~\eqref{eq:a_k} are two linear inequality constraints in 
	$a$. Thus,~\eqref{eq:a_k} is a strictly convex optimization 
	problem. Problem~\eqref{eq:a_k} is always feasible as the choice $a_0 
	\phi_0(0) = \hat{u}_i(0)$ and $a_i = 0$ for $i \in \{1, \ldots, 
	p\}$, which gives $g(\param_i(k), \tau)$ as the zero order hold 
	signal, satisfies the constraints. The final claim is now 
	obvious.
\end{proof}

\begin{rem}\thmtitle{Algebraic method to solve Problem~\eqref{eq:a_k}}
  \label{rem:algebraic-sol-to-optim-prob}
  While one may simply rely on standard optimization solvers to solve 
  Problem~\eqref{eq:a_k}, one may also choose to solve a set of 
  algebraic equations to obtain the unique solution to  it, specially 
  if the number of functions $p+1$ in 
  $\Phi$ is small. Moreover, the algebraic method provides useful 
  properties of the solutions of Problem~\eqref{eq:a_k} as a function 
  of $\hat{x}(t_k) = x(t_k)$. This is necessary for the analysis of the 
  overall \new{ETC} system.
  
  The Lagrangian corresponding to Problem~\eqref{eq:a_k} is 
  $\mathscr{L}_i(a,\mu)=J_i(a)+\mu_i^\top (y^\top a - z_i(k))$, 
  where $\mu_i \in \real^2$ is the Lagrange multiplier. Since the 
  Problem~\eqref{eq:a_k} is a strictly convex quadratic 
  program with two linear inequality constraints, strong duality holds 
  for problem~\eqref{eq:a_k} and any optimal primal-dual solution 
  $(\param_i(k),\mu_i(k))$ must satisfy the Karush-Kuhn-Tucker (KKT) 
  conditions. The stationarity conditions can be represented as,
  \begin{equation*}
      \textbf{H} \param_i(k)-b_i(k)+y\mu_i(k)=0.
  \end{equation*}
  Further, the complementary slackness conditions are
  \begin{equation*}
    \mu_{ij}(k)( y_j^\top \param_{i}(k)-z_{ij}(k)) = 0, \quad j \in 
    \{1, 2\},
     \end{equation*}
  where $y_j$ denotes the $j\tth$ column of $y$ for $j \in \{1,2\}$.
  
  Now, we have three different cases. Case 1: both the constraints 
  are inactive. Case 2: Only the first constraint is active. Case 3: 
  Only the second constraint is active. Note that, in 
  problem~\eqref{eq:a_k}, both the constraints can not be active at the 
  same time. In Case 1, $\param_i(k)=\textbf{H}^{-1}b_i(k)$ as 
  $\mu_i(k)=0$. In Case 2, $ y_1^\top \param_{i}(k)=z_{i1}(k)$ and $ 
  \textbf{H} \param_i(k)-b_i(k)+y_1\mu_{i1}(k)=0$ as $\mu_{i2}(k)=0$. 
  By using these facts, we can write,
  \begin{equation*}
    \begin{bmatrix}
      \param_i(k) \\ \mu_{i1}(k)
    \end{bmatrix} = \begin{bmatrix}
      \textbf{H} & y_1 \\ y_1^\top & 0
    \end{bmatrix}^{-1} \begin{bmatrix}
      b_i(k) \\ z_{i1}(k)
    \end{bmatrix} =: M \begin{bmatrix}
      b_i(k) \\ z_{i1}(k)
    \end{bmatrix}.
  \end{equation*}
  Note that the matrix $M$ depends only on 
  the set of functions $\Phi$ and not on $\hat{x}(t_k) = x(t_k)$.
   We can find a similar closed form expression of $\param_i(k)$ in 
  Case 3. In general, one can compute the candidate solutions for 
  each of the three cases and pick the one that satisfies the 
  corresponding constraints. \remend
\end{rem}

\subsection{Event-Triggering Rule}

We consider the following \new{ETR,} \new{which includes two 
conditions. 
The first one is a relative thresholding condition on the actuation 
error $e$ and the second condition helps to guarantee global uniform 
ultimate boundedness of the trajectories of the closed loop system 
under unknown disturbances.} 
\begin{equation}\label{eq:etr}
t_{k+1}= \min \{t>t_k \ : \rho_1(\norm{e}) \ge \frac{\sigma}{2} \alpha_3(\norm{x})  \text{ and }  V(x) \ge \epsilon\}, 
\end{equation}
where $e \new{=} u-\gamma(x)$, $\epsilon \ldef 
\alpha_2(\alpha_3^{-1}(\frac{2\rho_2(D)}{\sigma}))\ge0$ and $\sigma \in (0,1)$ is a design 
parameter. Here $\rho_1(.)$, 
$\rho_2(.)$, $\alpha_2(.)$, and $\alpha_3(.)$ are the same class $\Mc{K}_\infty$ functions 
given in Assumption~\ref{A:V}. Note that, here, $e$ denotes the error 
between the actual control input $u$ and the ``ideal'' feedback control 
input $\gamma(x)$. This error is different from the approximation error 
$u-\hat{u}$ as the dynamics of $x$ and $\hat{x}$ are different. 

In summary, the closed loop system, $\sys$, is the combination 
of the system dynamics~\eqref{eq:sys}, the control 
law~\eqref{eq:control_law}, with coefficients chosen by 
solving~\eqref{eq:a_ik}, which are updated at the events determined by 
the \new{ETR}~\eqref{eq:etr}. That is,
\begin{equation}\label{eq:full-system}
\sys : \ \eqref{eq:sys}, \eqref{eq:control_law}, 
\eqref{eq:a_ik}, \eqref{eq:etr}.
\end{equation}

\begin{rem}\thmtitle{Computational requirement of the controller}
We suppose that the controller has enough computational resources to 
evaluate the \new{ETR}~\eqref{eq:etr} and to solve the 
finite horizon optimization problem~\eqref{eq:a_k} at any triggering 
instant. Note that, \new{ET-}MPC or \new{ET-DBC} methods 
also have similar computational requirements at the controller.   
\remend
\end{rem}

\subsection{Analysis of the event-triggered control system}

Next, we show that the trajectories of the closed loop 
system~\eqref{eq:full-system} are globally uniformly ultimately bounded 
and the \new{IET}s have a uniform positive lower bound that 
depends on the initial state of the system. First, let
\begin{equation*}
  \epsilon_k \ldef V(x(t_k)), \ k \in \natz, \quad \bar{\epsilon} \ldef 
  \max\{\epsilon,\epsilon_0\}.
\end{equation*}
Following lemmas help to prove the main result of this paper.

\begin{lem}\label{lem:V}
	Consider system~\eqref{eq:full-system} and let Assumptions~\ref{A:V} 
	-~\ref{A:phis} hold. Let $\eta(.) \ldef 
	\frac{1}{\sqrt{m}}\rho_1^{-1}(r\frac{\sigma}{2} \alpha_3(.))$ with $r \in [0,1)$. Then, $V(x(t)) \le \epsilon_k \leq 
	\bar{\epsilon}, 
	\forall t \in [t_k,t_{k+1})$ and $\forall k \in 
	\nat$.
\end{lem}
\begin{proof}  
	Let us first calculate the time derivative of $V(.)$ along the trajectories of system~\eqref{eq:full-system} as follows,
	\begin{equation*}\label{eq:V_dot}
	\begin{aligned}
	\dot{V}&=\frac{\partial V}{\partial x}f(x,u,d)=\frac{\partial V}{\partial x} f(x,\gamma(x)+e,d),\\&
		\le -\alpha_3(\norm{x})+\rho_1(\norm{e})+\rho_2(D),\\
	& \le 	-(1-\sigma)\alpha_3(\norm{x})- \left( \frac{\sigma}{2} 
	\alpha_3(\norm{x})-\rho_1(\norm{e}) \right), \forall V(x) \ge 
	\epsilon.
	\end{aligned}
	\end{equation*}
	The first inequality follows from Assumptions~\ref{A:V} 
	and~\ref{A:d}. The last inequality follows from the 
	fact that $V(x)\ge \epsilon$ implies $\norm{x} \ge 
	\alpha_2^{-1}(\epsilon)=\alpha_3^{-1} \left( 
	\frac{2\rho_2(D)}{\sigma} \right)$. Note that \new{$e$ may be 
	discontinuous at $t_k$ as the control input $u$ is updated at $t_k$ 
	and} according to the inequality constraint 
	in~\eqref{eq:a_ik}, $\norm{e(t_k^{+})} \le \sqrt{m} 
	\eta(\norm{x(t_k^{+})}), \ \forall k \in \natz$.
	From the definition of $\eta(.)$ and the fact that $r \in [0,1)$, we 
	see that $\rho_1(\norm{e(t_k^{+})}) < \frac{\sigma}{2} 
	\alpha_3(\norm{x(t_k^{+})}), \ \forall k \in \natz$. Further, the 
	\new{ETR}~\eqref{eq:etr} implies that $V(x(t_k)) = 
	\epsilon_k \ge \epsilon, \ \forall k \in \nat$ and \new{as $x(t)$ is 
	continuous for all $t$,} $\dot{V}(x(t_k^{+})) \le 
	-(1-\sigma)\alpha_3(\norm{x(t_k^{+})})<0, \ 
	\forall k \in \nat$.
    
  Now, let us prove the statement that $V(x(t)) \le \epsilon_k, \  
  \forall t \in [t_k,t_{k+1})$ and $\forall k \in 
  \nat$ by contradiction. Suppose that this statement is 
  not true. Then, as $V(x(t))$ is a continuous function of time, there 
  must exist $\bar{t} \in (t_k,t_{k+1})$, for some $k \in \nat$, such 
  that $V(x(\bar{t}))=\epsilon_k$ and 
  $\dot{V}(x(\bar{t}))>0$. However, since $\epsilon_k \ge \epsilon$ and 
  the \new{ETR} is not satisfied at $t=\bar{t}$, 
  $\rho_1(\norm{e(\bar{t})}) < \frac{\sigma}{2} 
  \alpha_3(\norm{x(\bar{t})})$, which means that $\dot{V}(x(\bar{t})) 
  \le -(1-\sigma)\alpha_3(\norm{x(\bar{t})}) < 
  0$. As there is a contradiction, we conclude that there does not 
  exist such a $\bar{t}$ and hence $V(x(t)) \le \epsilon_k, \  
  \forall t \in [t_k,t_{k+1})$ and $\forall k \in 
  \nat$. 
  
  Finally, if $\epsilon_0 \le \epsilon$ then according to the 
  \new{ETR}~\eqref{eq:etr}, $\epsilon_1=\epsilon=\bar{ 
  \epsilon }$. If $\epsilon_0 > \epsilon$ then following similar 
  arguments as before we can show that $V(x(t)) \le \epsilon_0, \  
  \forall t \in [t_0,t_{1})$ and hence $\epsilon_1 \le \epsilon_0=\bar{ 
  \epsilon }$. Thus, the claim that $\epsilon_k \le \bar{\epsilon}, \ 
  \forall k \in \nat$ follows by induction.
\end{proof}

Note that Lemma~\ref{lem:V} does not impose any restrictions on 
$x(t_0)$. In particular, it is possible that $V(x(t_0)) < \epsilon$. 
Lemma~\ref{lem:V} only makes a claim about $V(x(t)), \ \forall t \in 
[t_k,t_{k+1}), \ \forall k \in \nat$. 

\begin{lem}\label{lem:u}
	Consider system~\eqref{eq:full-system} and let Assumptions~\ref{A:V} 
	-~\ref{A:phis} hold. Let $\eta(.) \ldef 
	\frac{1}{\sqrt{m}}\rho_1^{-1}(r\frac{\sigma}{2} \alpha_3(.))$ with $r \in [0,1)$. Then, there exist $\beta_1, \beta_2 >0$ such that 
	$\norm{u(t)} \le \beta_1$ and $\norm{\dot{u}(t)} \le 
	\beta_2$, $\forall t \in [t_k, \min\{ t_{k+1}, t_k+T \})$, 
	$\forall k\in \natz$.
\end{lem}
\begin{proof}
	Recall that, $\forall i \in 
	\{1,2,\ldots,m\}$ and $\forall k \in \natz$, $u_i(t)$ for $t \in 
	[t_k, t_{k+1})$ is chosen by solving the problem~\eqref{eq:a_ik}, 
	which is equivalent to~\eqref{eq:a_k}.	
	Now, 
		we consider the compact set $R \ldef \{x \in \real^n:V(x) \le 
		\bar{\epsilon}\}$. According to Lemma~\ref{lem:V}, for any $k \in \natz$, 
	$\hat{x}(t_k)=x(t_k) \in R$, which implies that $\hat{x}(t) \in R$ for 
	all $t \in [t_k,t_k+T)$ as $\dot{V}(\hat{x})\le-\alpha_3(\hat{x})\le0$. 
	Now, note that,
	\begin{equation*}
	| \hat{u}_i(\tau) | \le \norm{\gamma(\hat{x}(t_k+\tau))} \le L_{\gamma},
	\end{equation*}
	for some $L_{\gamma}>0$. The last inequality follows from the 
	fact that $\gamma(.)$ is Lipschitz on the compact set $R$ with 
	$\gamma(0)=0$ and $\norm{\hat{x}} \le \alpha_1^{-1}( \bar{ \epsilon } 
	)$ for all $\hat{x} \in R$. 
The fact $\hat{x}(t_k)=x(t_k) \in R$, $ \forall 
  k \in \natz$, further implies that $\eta(\norm{ x(t_k) })$ and 
  $|\hat{u}_i(0)|$ are upper bounded, $\forall k \in \natz$ and thus, 
  $\norm{b_i(k)}$ and $\norm{z_i(k)}$ are also uniformly, $\forall k 
  \in \natz$, upper bounded by some constants.
  This along with the algebraic solution of $\param_i(k)$ given in 
  Remark~\ref{rem:algebraic-sol-to-optim-prob}, implies that 
  $\norm{\param_i(k)}$ is upper bounded by a constant for each $i \in 
  \{1,2,..,m\}$, $\forall k \in \natz$.
  Since each $\phi_j(.) \in \Phi$ is continuously differentiable on 
  $[0, T]$, we can say that there exist $\beta'_1, \beta'_2 > 0$ such 
  that 
  $\norm{\phi(t-t_k)} \le \beta'_1$ and $\norm{ \dfrac{\dd}{\dd t} 
  \phi(t-t_k)} \le \beta'_2$, $\forall t \in [t_k,\min\{ t_{k+1}, t_k+T 
  \} ], \ \forall k \in \natz$. Putting it all together along 
	with~\eqref{eq:control_law} proves the result.
\end{proof}
Now, we present the main 
result of this paper.
\begin{thm}\label{thm:analysis}\thmtitle{Absence of Zeno behavior and global uniform ultimate boundedness}
	Consider system~\eqref{eq:full-system} and let Assumptions~\ref{A:V} 
	-~\ref{A:phis} hold. Let $\eta(.) \ldef 
	\frac{1}{\sqrt{m}}\rho_1^{-1}(r\frac{\sigma}{2} \alpha_3(.))$.
	\begin{itemize}
		\item If 
		$r \in \left[ 0, \frac{ \alpha_3 ( \alpha_2^{-1} (\epsilon) 
		}{ \alpha_3( \alpha_1^{-1}(\bar{\epsilon}))} \right)$, then the 
		\new{IET}s, $t_{k+1} - t_k, \ \forall k \in \nat$, 
		are uniformly lower bounded by a positive number that depends 
		on the initial state of the system.
		\item The trajectories of the closed loop system are globally uniformly ultimately bounded with global uniform ultimate bound $\alpha_1^{-1}(\epsilon)$.
	\end{itemize}

\end{thm}

\begin{proof}
	Let us prove the first statement of this theorem. First, note that 
	Assumption~\ref{A:V} implies $\alpha_1(\norm{x}) \leq 
	\alpha_2(\norm{x})$ $\forall x \in \real^n$ and then by the 
	definition of $\bar{\epsilon}$, we can say that $r \in [0, 1)$. Now, 
	we consider the compact set $R \ldef \{x \in \real^n:V(x) \le 
	\bar{\epsilon}\}$. Lemma~\ref{lem:V} implies that $x(t) \in R, \ 
	\forall t \in [t_k,t_{k+1}), \forall k \in \nat$. Note that according 
	to the proof of Lemma~\ref{lem:V},
$\norm{e(t_k^{+})} \le \rho_1^{-1}(r\frac{\sigma}{2} 
\alpha_3(\norm{x(t_k^{+})})), \ \forall 
	k \in \natz$. Now, by using the fact that $\alpha_2^{-1} (\epsilon) 
	\le \norm{x(t_k)}\le\alpha_1^{-1}( \bar{ \epsilon } ), \forall k \in 
	\nat$,  
	we can say that the \new{IET} must at least be equal to the 
	time it takes $\norm{e}$ to grow from $e_1 
	\ldef\rho_1^{-1}(r\frac{\sigma}{2} \alpha_3(\alpha_1^{-1}(\bar{  
	\epsilon} )))$ to $e_2 \ldef \rho_1^{-1}(\frac{\sigma}{2} 
	\alpha_3(\alpha_2^{-1}(\epsilon)))$. Note that, if $r$ is chosen as 
	in the statement of the result, then we can guarantee that $e_1 < 
	e_2$. Next, we give a uniform upper-bound on $D^+ 
	\norm{e(t)}, \ \forall t \in [t_k,t_{k+1}), \forall k \in \nat$. 
  Note that $\forall i \in \{1, \ldots, m\}$,
  \begin{align*}
    | D^+ \gamma_i(x(t)) | &\leq \limsup_{h \rightarrow 0^+} \frac{| 
    \gamma_i(x(t+h)) - \gamma_i(x(t)) |}{h}
  \\
  &\leq \limsup_{h \rightarrow 0^+} \frac{L \norm{ x(t+h) - x(t) } 
  }{h}
  \\
  &\leq L \limsup_{h \rightarrow 0^+} \norm{f(x(t+h), u(t+h), d(t+h)} 
  \leq \beta_3,
  \end{align*}
  where we have used the fact $x(t) \in R, \ \forall t \in 
  [t_k,t_{k+1}), \forall k \in \nat$ from Lemma~\ref{lem:V}, and $L$ is 
  a Lipschitz constant for $\gamma_i$ on the compact set $R$ and the 
  final inequality follows from the additional facts of boundedness of 
  $\norm{u(t)}$ from Lemma~\ref{lem:u} and Assumptions~\ref{A:sys} 
  and~\ref{A:d}. We then have
  \begin{equation*}
    D^+ \norm{e} \le \norm{\dot{u}} + m \beta_3 \leq \beta, \quad 
    \forall t \in [t_k,t_{k+1}), \forall k \in \nat,
  \end{equation*}
    
	for some $\beta>0$. The last inequality follows from the uniform 
	boundedness of $\norm{\dot{u}}$ from Lemma~\ref{lem:u}. This implies 
	that $t_{k+1}-t_k \ge 
	\frac{e_2-e_1}{\beta}$, for any $k \in \nat$, which completes the 
	proof of the first statement of this result.
  
  Next, since the \new{IET}s have a uniform positive lower bound, $t_k 
	\rightarrow \infty$ as $k \rightarrow \infty$. Thus, 
	for all $t \ge t_0 = 0$,
		\begin{equation*}
		\dot{V} \le  -(1-\sigma)\alpha_3(\norm{x})<0, \quad \forall V(x) 
		\ge \epsilon,
	\end{equation*}
  \new{which implies the second statement of the result.}
\end{proof}

\section{NUMERICAL EXAMPLES}\label{sec:numerical_examples}

We illustrate our results with two numerical examples. \newline
\textbf{Example 1:} (Controlled Lorenz model with 
disturbances)
\begin{align}
&\dot{x}_1=-ax_1+ax_2+d_1, \notag\\
&\dot{x}_2=bx_1-x_2-x_1x_3+u+d_2, \notag\\
&\dot{x}_3=x_1x_2-cx_3+d_3, \label{eq:controlled-lorenz}
\end{align}
where $a,b,c \in \realp$ and $d \ldef \begin{bmatrix}
d_1 & d_2 & d_3
\end{bmatrix}^{\top}$ is the external disturbance. We set the parameter values $a=10$, $b=28$ and $c=8/3$. We can show that Assumption~\ref{A:V} and Assumption~\ref{A:sys} hold with $V(x)=\frac{1}{2}\norm{x}^2$ and $\gamma(x)=-(a+b)x_1-\frac{1}{2}x_2$. We consider the disturbance $d(t)=\frac{0.1}{\sqrt{3}}\begin{bmatrix}
	\sin(50t) &  \sin(20t) &  \sin(10t)
	\end{bmatrix}^{\top}.$ Note that Assumption~\ref{A:d} holds with $D=0.1$.  In this example, we consider the control input as 
a linear combination of the set of functions $\{1,\tau, 
\tau^2,\ldots,\tau^p\}$. 
\begin{figure}[h]
	\centering
	\begin{subfigure}{4cm}
		\includegraphics[width=4.4cm]{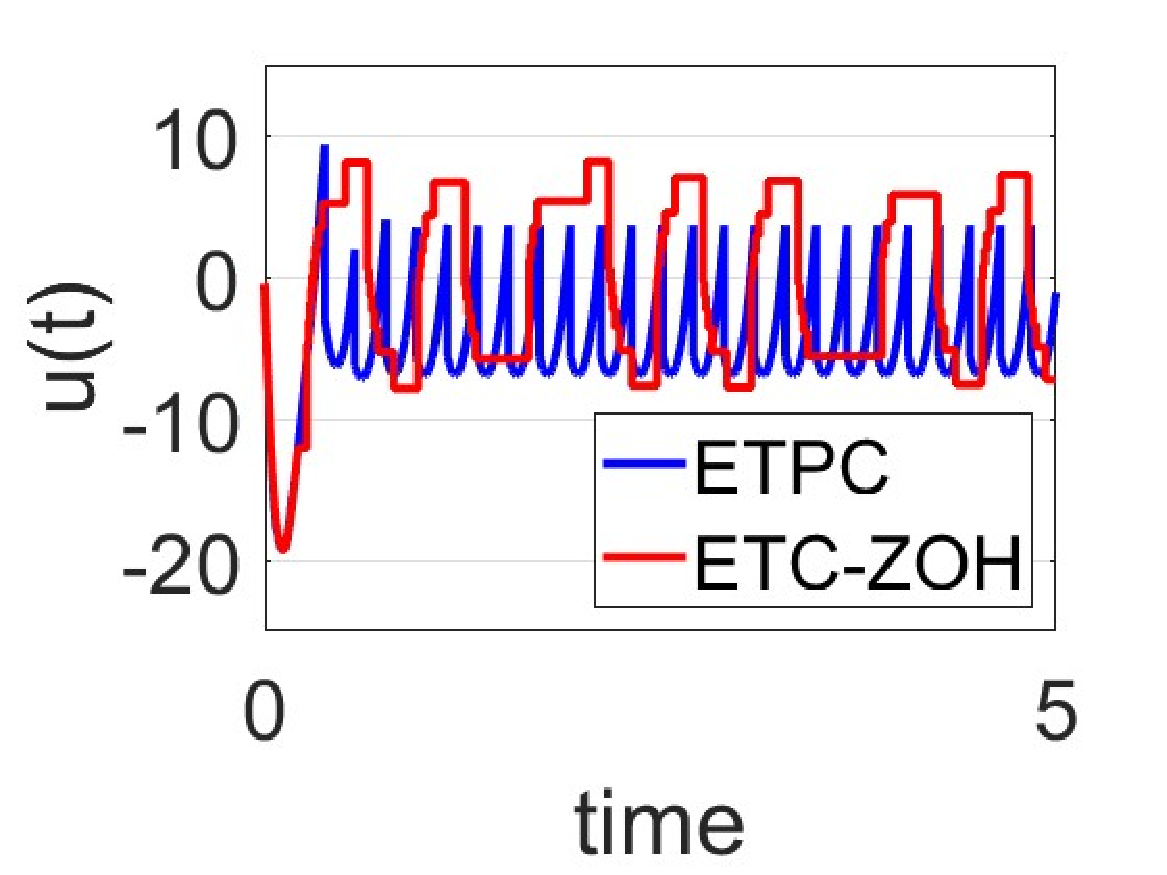}
		\caption{Control input signal}
		\label{fig:IET1}
	\end{subfigure}
	\quad
	\begin{subfigure}{4cm}
		\includegraphics[width=4.4cm]{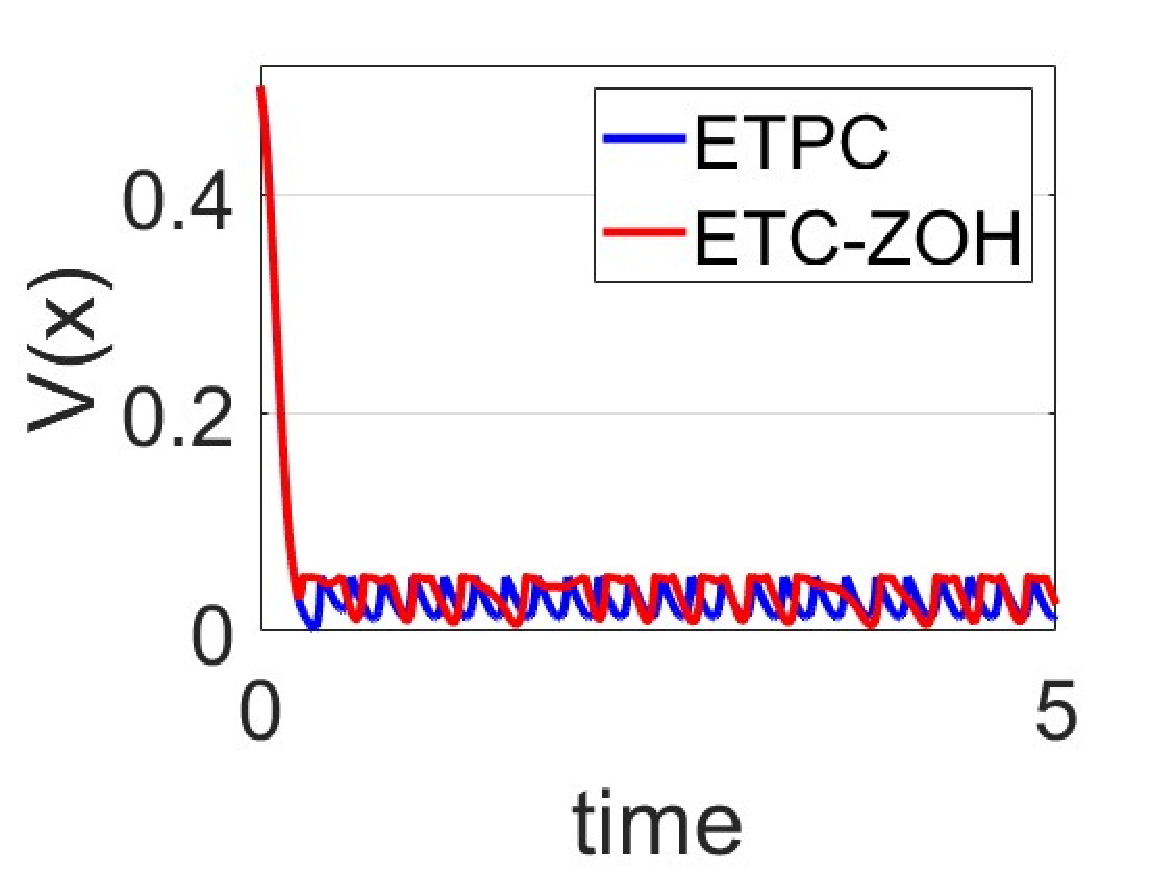}
		\caption{Evolution of $V(x(t))$}
		\label{fig:V1}
	\end{subfigure}
	\caption{Simulation results of Example 1 for $p=3$, 
      $T=0.1$ and $x(0)=[0\quad1\quad0]^{\top}$.}
	\label{fig:Ex1}
\end{figure}
\new{Here, we compare the performance of the proposed ETPC method with 
the zero-order-hold control based ETC (ETC-ZOH) method. In ETC-ZOH 
method, we choose the same \new{ETR}~\eqref{eq:etr} without further 
tuning of the parameters, but the control law is $u(t)=\gamma(x(t_k)), 
\ \forall t \in [t_k,t_{k+1})$.}
Figure~\ref{fig:Ex1} presents the simulation results with $p=3$, 
$T=0.1$ and $x(0)=[0\quad1\quad0]^{\top}$. Figure~\ref{fig:IET1} 
presents the evolution of control input to the plant for the proposed ETPC method and 
zero-order-hold ETC (ETC-ZOH). Figure~\ref{fig:V1} 
shows the evolution of $V(x)$ along the system trajectory for both the 
methods. Note that, $V(x)$ converges to the ultimate bound 
$\epsilon=0.05$ in both the cases.

Next, we consider $100$ initial conditions uniformly sampled from the 
unit sphere and we calculate the average \new{IET} (AIET) and 
the minimum \new{IET} (MIET) over $100$ events for each initial 
condition with $T=0.3$, and $p=3$. The average of AIET over the set of 
initial conditions is observed as $0.0138$ and $0.3649$ for 
ETC-ZOH and ETPC, respectively. The minimum of MIET over the set of 
initial conditions is observed as $8.2744 \times 10^{-4}$ and $0.0141$ 
for ETC-ZOH and ETPC, respectively. 
\new{Note that, for the given choice of control law and the \new{ETR}, the proposed ETPC method performs better, in terms of the AIET and MIET, compared to the ETC-ZOH method.}

We 
repeat the procedure for different values of 
$T$ and $p$, and the observations are tabulated in 
Table~\ref{tab:table2}. In Table~\ref{tab:table2}, we can see that 
there is a decreasing trend in the values of AIET and MIET as 
$T$ increases. Note that choosing a larger $T$ can lead to a 
	greater fitting error in the short-term just after $t_k$ and hence 
	leads to smaller \new{IET}s. Whereas, there is an increasing trend in AIET and MIET as $p$ increases. Note that choosing a larger $p$ helps to find a better approximation of the continuous time model based control signal and hence leads to larger \new{IET}s.

\begin{table}[h]
	\begin{center}
		\caption{Average of AIET and minimum of MIET, over a set of initial conditions, for 
				different values of $T$ and $p$.}
		\label{tab:table2}
		\begin{tabular}{|c|c|c|c|c|c|c|} 
			\hline
			& \multicolumn{6}{c}{T}\vline\\
			\hline
			& \multicolumn{2}{c}{0.4} \vline& \multicolumn{2}{c}{0.6} \vline& \multicolumn{2}{c}{0.8}\vline\\
			\hline
			p & \textbf{AIET} & \textbf{MIET} & \textbf{AIET} & \textbf{MIET} & \textbf{AIET} &\textbf{MIET} \\
			\hline
			3 & 0.2552 & 0.0073 & 0.1389 & 0.0036 & 0.0919 & 0.0024\\
			4 & 0.3686 & 0.0209 & 0.2691 & 0.0068 & 0.1662 & 0.0042\\
			5 & 0.3698 & 0.0224 & 0.3668 & 0.0166 & 0.3043 & 0.0071\\
			\hline
		\end{tabular}
	\end{center}
\end{table}

Next, \new{for the system~\eqref{eq:controlled-lorenz} with no} 
disturbance, 
we compare the performance of the ETPC method against dynamic 
\new{ETC} (DETC) with the following 
dynamic \new{ETR} proposed in~\cite{AG:2015},
	\begin{equation}\label{eq:DETC}
	t_{k+1}= \min \{t>t_k \ :  \nu(t) + \theta \left( \frac{\sigma}{2} 
	\alpha_3(\norm{x})-\rho_1(\norm{e}) \right) \le 0\}, 
	\end{equation}
	where the dynamic variable $\nu(t)$ follows the dynamics,
	\begin{equation*}
	\dot{\nu}(t)= -\omega(\nu)+\frac{\sigma}{2} \alpha_3(\norm{x})-\rho_1(\norm{e}), \ \nu(0)=\nu_0.
	\end{equation*}
	Here, $\omega(.)$ is a class $\Mc{K}_{\infty}$ function, $\nu_0 \ge 0$ and $\theta \ge 0$ are design parameters. We choose $\omega(\nu)=0.5\nu$ and $\nu_0=0$. In Table~\ref{tab:table3}, we compare the performance of DETC method based on ZOH control, ETPC method with static \new{ETR}~\eqref{eq:etr} and ETPC method with dynamic \new{ETR}~\eqref{eq:DETC}.
	\begin{table}[h]
		\begin{center}
			\caption{Average of AIET and minimum of MIET, over a set of 
			100 initial conditions for DETC and ETPC.}
			\label{tab:table3}
			\begin{tabular}{|c|c|c|c|} % <-- Alignments: 1st column left, 2nd middle and 3rd right, with vertical lines in between
				\hline &
				\textbf{DETC-ZOH} & 
				\textbf{ETPC-static} & 
				\textbf{ETPC-dynamic}\\
				\hline
				\textbf{AIET} & 0.0066  & 0.0534 & 0.2550  \\
				\hline
				\textbf{MIET} & $8.7633 \times 10^{-4}$  & 0.0010 &	0.0191 \\
				\hline
			\end{tabular}
		\end{center}
	\end{table}
	We can see that, \new{for the given choice of the dynamic variable $\nu$}, the proposed ETPC method (with $T=0.3$ and $P=5$) \new{performs better in terms of the AIET and MIET} compared to the \new{DETC} method. \newline
\textbf{Example 2:} (Forced Van der Pol oscillator with 
disturbances)
	\begin{align*}
	&\dot{x}_1=x_2+d_1, \quad \dot{x}_2=(1-x_1^2)x_2-x_1+u+d_2,
	\end{align*}
	where $d \ldef \begin{bmatrix}
	d_1 & d_2 
	\end{bmatrix}^{\top}$ is the external disturbance. We can show that Assumption~\ref{A:V} and Assumption~\ref{A:sys} hold with $V(x)=x^{\top}Px$ where $P=\begin{bmatrix}
	4.5 & 1.5 \\ 1.5 & 3
	\end{bmatrix}$ and $\gamma(x)=-x_2-(1-x_1^2)x_2$. We consider the disturbance $d(t)=\frac{0.1}{\sqrt{2}}\begin{bmatrix}
	\sin(10t) &  \sin(20t)
	\end{bmatrix}^{\top}.$ Note that Assumption~\ref{A:d} holds with $D=0.1$.  In this example also, we consider the control input as 
	a linear combination of the set of functions $\{1,\tau,\ldots,\tau^p\}$. The average of AIET over a set of 100
	initial conditions is observed as $0.0938$ and $1.6948$ for 
	ETC-ZOH and ETPC (with $T=1$ and $p=3$), respectively. The minimum of MIET over the set of 
	initial conditions is observed as $9.93 \times 10^{-4}$  and  $0.0096$
	for ETC-ZOH and ETPC, respectively.
	
\section{CONCLUSION}\label{sec:conclusion}

In this paper, we proposed the \new{ETPC} method for control of nonlinear systems with external 
disturbances. 
We designed a parameterized control law and an \new{ETR} 
that guarantee global uniform ultimate boundedness of the trajectories 
of the closed loop system and non-Zeno behavior of the generated 
\new{IET}s. 
\new{We illustrated our 
results through numerical examples.} 
Future work includes the generalization of this control method to 
distributed control setting, analytical
	method to determine an optimal time horizon for function
	fitting, systematic methods for choosing the basis functions, 
	control under model uncertainty, quantization of the parameters, time 
	delays, and a control Lyapunov function or MPC approach to 
ETPC.

\bibliographystyle{IEEEtran}
\bibliography{references}
\end{document}